\documentclass{amsart}

\usepackage{amsfonts,amssymb,amscd,amsmath,enumerate,verbatim,newlfont,calc}
\usepackage{graphicx}
\RequirePackage[all]{xy}

\newtheorem{theorem}{Theorem}[section]
\newtheorem{theorem A}{Theorem A}
\newtheorem{theorem B}{Theorem B}
\newtheorem{lemma}[theorem]{Lemma}

\newtheorem{remark}[theorem]{Remark}

\newtheorem{proposition}[theorem]{Proposition}
\theoremstyle{definition}
\newtheorem{definition}[theorem]{Definition}
\begin{document}
\address{Department of Mathematics, Ferdowsi University of Mashhad, Mashhad, Iran}
\title{$\tilde{\mathcal{B}}_0$-invariant of groups}
\author[Z. Araghi Rostami]{Zeinab Araghi Rostami}
\author[M. Parvizi]{Mohsen Parvizi}
\author[P. Niroomand]{Peyman Niroomand}
\address{Department of Pure Mathematics\\
Ferdowsi University of Mashhad, Mashhad, Iran}
\email{araghirostami@gmail.com, zeinabaraghirostami@alumni.um.ac.ir}
\address{Department of Pure Mathematics\\
Ferdowsi University of Mashhad, Mashhad, Iran}
\email{parvizi@um.ac.ir}
\address{Department of Mathematics, Damghan University, Damghan, Iran}
\email{niroomand@du.ac.ir}
\maketitle
\begin{abstract}
The Bogomolov multiplier of a group \( G \) was introduced by Bogomolov in \( 1988 \). Afterward, in \( 2012 \), Moravec introduced an equivalent definition of the Bogomolov multiplier. In this paper, we generalize the Bogomolov multiplier with respect to a variety of groups. Subsequently, we present some new results on this topic.
\end{abstract}
\section{\bf{Introduction}}

For a group \( G \), by a result of Hopf in \cite{7}, if \( G \cong \frac{F}{R} \) is a free presentation of \( G \), then the groups \( \frac{F'}{[R,F]} \) and \( \frac{F'\cap R}{[R,F]} \) are independent of the choice of \( F \), i.e., are invariants of \( G \). Moreover, he proved that \( \frac{F'\cap R}{[R,F]} \) is isomorphic to the Schur multiplier \( H^2(G,\mathbb{Q}/\mathbb{Z}) \) of \( G \). Baer in \cite{2} generalized these concepts to an arbitrary variety of groups. Later, Fröhlich in \cite{4} generalized the idea of Baer for any two varieties \( \mathcal{V} \) and \( \mathcal{W} \) when \( \mathcal{W} \subseteq \mathcal{V} \), and in the third article \cite{9'}, the Baer invariant of \( G \) was reviewed with respect to two varieties \( \mathcal{V} \) and \( \mathcal{W} \) in general.

Bogomolov in \cite{3} proved that the unramified cohomology group \( H_{nr}^2(\mathbb{C}(V)^G,\mathbb{Q}/\mathbb{Z}) \) is canonically isomorphic to
\[ \mathcal{B}_0(G) = \bigcap_{A\leq G} \text{ker} \, \text{res}_A^G, \]
where \( A \) is an abelian subgroup of \( G \), \( V \) is a faithful representation of \( G \) over \( \mathbb{C} \), \( \text{res}_A^G: H^2(G,\mathbb{Q}/\mathbb{Z}) \longrightarrow H^2(A,\mathbb{Q}/\mathbb{Z}) \) is the usual cohomological restriction map, and \( \mathbb{C}(V)^G \) is the field of \( G \)-invariant functions. \( \mathcal{B}_0(G) \) is defined as the subgroup of the so-called Schur multiplier \( \mathcal{M}(G) = H^2(G,\mathbb{Q}/\mathbb{Z}) \) of \( G \). Kunyavskii in \cite{8} named \( \mathcal{B}_0(G) \) the \emph{Bogomolov multiplier} of \( G \). Moravec in \cite{10} showed that in the class of finite groups, the group \( \tilde{\mathcal{B}}_0(G) = \mathcal{M}(G)/\mathcal{M}_0(G) \) is non-canonically isomorphic to \( \mathcal{B}_0(G) \), where the Schur multiplier \( \mathcal{M}(G) \) is isomorphic to \( \text{ker} (G\wedge G \rightarrow [G,G]) \) given by \( x\wedge y \mapsto  [x,y] \) and \( \mathcal{M}_0(G) \) is a subgroup of \( \mathcal{M}(G) \) defined as \( \mathcal{M}_0(G) = \langle x\wedge y \ | \  [x,y]=0 , \  x,y\in G \rangle \). Moravec in \cite{10} showed that if \( \frac{F}{R} \) is a free presentation of the group \( G \), then \( \tilde{\mathcal{B}}_0(G) \cong \frac{F'\cap R}{\langle K(F)\cap R \rangle} \) and it is an invariant of \( G \), where \( K(F) =\{ [x,y] \ | \ x,y\in F\} \).

In this paper, we are going to generalize the concept of the Bogomolov multiplier with respect to an arbitrary variety of groups. Here, we want to introduce a large class of invariants of \( \frac{F'\cap R}{\langle K(F)\cap R \rangle} \), derived from free presentations.

\section{\bf{Notations and preliminaries}}

In the following, we provide some preliminaries required for subsequent discussions. 

Let \( x \) and \( y \) be two elements of a group \( G \). Then \( [x, y] \), the commutator of \( x \) and \( y \), and \( x^y \) denote the elements \( x^{-1}y^{-1}xy \) and \( y^{-1}xy \), respectively.

The left-normed commutators of higher weights are defined inductively as
\[ [x_1 , x_2 , \ldots , x_n] = [[x_1, \ldots ,x_{n-1}] , x_n]. \]

If \( M \) and \( N \) are two subgroups of a group \( G \), then \( [M,N] \) denotes the subgroup of \( G \) generated by all the commutators \( [m,n] \) with \( m \in M \) and \( n \in N \). In particular, if \( M = N = G \), then \( [G,G] \), denoted by \( G' \), is the derived subgroup of \( G \). The lower and upper central series are denoted by \( \gamma_{c}(G) \), where \( c \geq 1 \), and \( Z_{d}(G) \) where \( d \geq 0 \), respectively.

Let \( F_{\infty} \) be the free group freely generated by an infinite countable set \( \{x_1, x_2, \ldots \} \). If \( u = u(x_1, \ldots ,x_s) \) and \( v = v(x_1, \ldots ,x_t) \) are two words in \( F_{\infty} \), then the composite \( u \circ v \) of \( u \) and \( v \) is defined as 
\[ u \circ v = u(v(x_1, \ldots ,x_t), \ldots ,v(x_{(s-1)t+1}, \ldots ,x_{st})). \]

In particular, the composite of some nilpotent words is called a polynilpotent word, i.e.,
\[ \gamma_{c_1+1, \ldots ,c_t+1} = \gamma_{c_1+1} \circ \gamma_{c_2+1} \circ \ldots \circ \gamma_{c_t+1}, \]
where \( \gamma_{c_i+1} (1 \leq i \leq t) \) is a nilpotent word in distinct variables. Outer commutator words are defined inductively as follows. The word \( x_i \) is an outer commutator word (or o.c word) of weight one. If \( u = u(x_1, \ldots ,x_s) \) and \( v = v(x_{s+1}, \ldots ,x_{s+t}) \) are o.c words of weights \( s \) and \( t \) respectively, then 
$$w(x_1, \ldots ,x_{s+t}) = [u(x_1, \ldots ,x_s),v(x_{s+1}, \ldots ,x_{s+t})]$$
is an o.c word of weight \( s+t \). 

Let \( V \) be a subset of \( F_{\infty} \) and \( \mathcal{V} \) be the variety of groups defined by the set of laws \( V \). P. Hall in \cite{5} introduced the notion of the verbal subgroup \( V(G) \) and the marginal subgroup \( V^{*}(G) \), associated with the variety \( \mathcal{V} \) and the group \( G \), where 
\[ V(G) = \langle v(g_1,g_2,\ldots ,g_r) \ | \ v \in V ,  g_1,g_2,\ldots,g_r \in G \rangle \]
and

\begin{align*}
&V^{*}(G)=\\
&\{a\in G \ | \ v(g_1,g_2,\ldots ,g_ia, \ldots , g_r) =v(g_1,g_2,\ldots ,g_r) , v\in V ,  g_1,g_2,\ldots,g_r\in G \}.
\end{align*}

(see \cite{9,11}, for more information).

\begin{proposition} [\cite{9}, Proposition 1.3] \label{p2.1}
If \( N \unlhd G \) and \( \mathcal{V} , \mathcal{W} \) are varieties such that \( \mathcal{V} \subseteq [\mathcal{W},\mathcal{F}] \), where \( \mathcal{F} \) is the variety whose groups are all trivial, then
\renewcommand {\labelenumi}{(\roman{enumi})}
\begin{enumerate}
\item \( N \cap V(G) \supseteq [NV^*G] \supseteq [N,W(G)] \),
\item \( [V^*(G),W(G)] = 1 \),
\end{enumerate}
where \\
\( [NV^{*}G] = \langle v(g_1, \ldots ,g_i{n}, \ldots ,g_s)v(g_1, \ldots ,g_s)^{-1} \ | \ v \in V, g_i \in G, n \in N, 1 \leq i \leq s \rangle. \)

We remark that if \( \mathcal{V} \) is the variety of groups defined by o.c words \( V \), then for any group \( G \)
\begin{align*}
V^{*}(G) &= \{a \in G \ | \  v(g_1,\ldots , g_{i-1},a,g_{i+1},\ldots , g_s) = 1 ,  v \in V , g_1, \ldots ,g_s \in G , 1 \leq i \leq s \}.
\end{align*}
Hence, if \( N \) is a normal subgroup of \( G \) then
\[ [NV^{*}G] = \langle v(g_1,\ldots , g_{i-1},b,g_{i+1},\ldots , g_s)\ | \  v \in V , g_1, \ldots ,g_s \in G , 1 \leq i \leq s , b \in N \rangle. \]

See Lemma \( 2.3 \) in \cite{9''}.
\end{proposition}

\begin{lemma} [\cite{7''}, Lemma 2.9] \label{l2.2}
Let \( G \) be a group and \( u \), \( w \) be two words, \( v = [u,w] \) and \( K \) is a normal subgroup of a group \( G \), then
\[ [Kv^{*}G] = [[Ku^{*}G],w(G)][u(G), [Kw^{*}G]]. \]
\end{lemma}

\begin{proposition} \label{p2.3}
Let \( N \unlhd G \), and \( V, W \) be sets of o.c words and \( \mathcal{V}, \mathcal{W} \) are the corresponding varieties such that \( \mathcal{V} \subseteq [\mathcal{W},\mathcal{F}] \), where \( \mathcal{F} \) be the variety of the trivial group. Then
\[ [N,W(G)] \subseteq [NV^*G] \subseteq \langle T(G)\cap N \rangle \subseteq N\cap V(G), \]
where \( T(G) = \{v(g_1, \ldots ,g_s) \ | \ v \in V, g_i \in G ,1 \leq i \leq s\} \).
\end{proposition}

\begin{proof}
If \( v = [u,w] \), then 
\[V(G) = [U(G),W(G)]  \ \ , \ \  T(G) = \{[u(g_1,\ldots ,g_t),w(g_{t+1},\ldots ,g_{t+s})]\ | \ g_1,\ldots , g_{t+s}\in G\} .\]
 Hence, by using Proposition \ref{p2.1} and Lemma \ref{l2.2}, \( [NV^*G] \) is generated by the following set
\begin{align*}
&\{[u(g_1,\ldots ,g_{i-1},b,g_{i+1},\ldots ,g_t),w(g_{t+1},\ldots ,g_{t+s})]\ | \ \\
&g_1,\ldots ,g_{t+s}\in G, 1\leq i\leq t, b\in N\}\\
&\cup \{[u(g_1,\ldots ,g_t),w(g_{t+1},\ldots ,g_{t+j-1} ,b,g_{t+j+1},\ldots ,g_{t+s})]\ | \ \\
&g_1,\ldots ,g_{t+s}\in G, 1\leq j\leq s, b\in N\}.
\end{align*}
Thus, \( [N,W(G)] \subseteq [NV^*G] \subseteq \langle T(G)\cap N \rangle \subseteq N\cap V(G) \), and the result is obtained.
\end{proof}

Let \( G \) be a group with a free presentation \( G\cong {F}/{R} \), then the Baer invariant of \( G \) with respect to the variety \( \mathcal{V} \), is defined to be
\[ \mathcal{V}\mathcal{M}(G) = \frac{R\cap V(F)}{[RV^{*}F]}. \]

Leedham-Green and Mckay in \cite{9} showed that the Baer invariant of the group \( G \) is always abelian and independent of the choice of a free presentation of \( G \). In particular, if \( \mathcal{V} \) is the variety of abelian or nilpotent groups of class at most \( c \) \((c \geq 2)\), then the Baer invariant of \( G \) is isomorphic to the Schur multiplier of \( G \) or \( \frac{R \cap {\gamma}_{c+1}(F)}{[R,cF]} \) where \( ([R,cF]=[R,\underbrace{F,...,F}_{c-\text{copies}}]) \), respectively.

\section{$\tilde{\mathcal{B}}_0$-invariant}

In this section, we introduce another invariant of $G$ closely related to the Baer invariant.

\begin{definition}\label{d3.1}
Let $G$ be any group with a free presentation $G\cong {F}/{R}$ and $V$ be a set of o.c words. Then the $\tilde{\mathcal{B}}_0$-invariant $\mathcal{V}\tilde{\mathcal{B}}_0(G)$ of $G$ with respect to the variety $\mathcal{V}$ is defined by
$$\mathcal{V}\tilde{\mathcal{B}}_0(G)=\frac{R\cap V(F)}{\langle T(F)\cap R \rangle},$$
where $T(F)=\{v(f_1, \ldots ,f_s) \ | \ v\in V, f_i\in F, 1\leq i\leq s\}$.
\end{definition}

Since $[V(G),G]\subseteq V(G)$, therefore $\mathcal{V}\subseteq [\mathcal{V}, \mathcal{F}]$. Thus by using Proposition \ref{p2.3}, we have
$$[R\cap V(F), R\cap V(F)]\subseteq [R,V(F)]\subseteq [RV^*F]\subseteq  \langle T(F)\cap R \rangle.$$
Hence the group $\mathcal{V}\tilde{\mathcal{B}}_0(G)$ is abelian. Also we can write   
$$\mathcal{V}\tilde{\mathcal{B}}_0(G)=\frac{R\cap V(F)}{\langle T(F)\cap R \rangle }=\frac{(R\cap V(F))/[RV^{*}F]}{\langle T(F)\cap R \rangle /[RV^{*}F]}.$$
We know $\dfrac{(R\cap V(F))}{[RV^{*}F]}$ is the Baer invariant. Now, if we denote $\dfrac{\langle T(F)\cap R \rangle }{[RV^{*}F]}$ by $\mathcal{V}{\mathcal{M}}_{0}(G)$, then we have
$$\mathcal{V}\tilde{\mathcal{B}}_0(G)=\frac{\mathcal{V}\mathcal{M}(G)}{\mathcal{V}{\mathcal{M}}_{0}(G)}.$$
The next lemma shows this notion is independent of the chosen free presentation of $G$.

\begin{lemma}\label{l3.2}
Let $\mathcal{V}$ be a variety of groups, $G, \bar{G}$ be two groups with free presentations $1 \xrightarrow{} R \xrightarrow{}  F \xrightarrow{\pi} G \xrightarrow{} 1$, $1\xrightarrow{} \bar{R}\xrightarrow{} \bar{F} \xrightarrow{\bar{\pi}}  \bar{G}\xrightarrow{} 1$, respectively, and $\alpha : G\rightarrow \bar{G}$ be a homomorphism. Then
\renewcommand {\labelenumi}{(\roman{enumi})} 
\begin{enumerate}
\item{$\text{there exists a homomorphism} \ \beta : F\rightarrow \bar{F} \   \text{such that}\   {\alpha}{\pi}={\bar{\pi}}{\beta}$,}
\item{$\text{there exists a homomorphism} \  {\alpha}^* : \mathcal{V}\tilde{\mathcal{B}}_0(G) \rightarrow \mathcal{V}\tilde{\mathcal{B}}_0(\bar{G})\   \text{which only depends on} \ \alpha$,}
\item{if $1\rightarrow \hat{R}\rightarrow \hat{F}\rightarrow \hat{G} \rightarrow 1$ is a free presentation for $\hat{G}$ and $\gamma : \bar{G}\rightarrow \hat{G}$ is a homomorphism, then $({\gamma}{\alpha})^*={\gamma}^*{\alpha}^*$,}
\item{$\mathcal{V}\tilde{\mathcal{B}}_0(G)$ is independent of the presentation of $G$.}
\end{enumerate}
\end{lemma}

\begin{proof}
$(i)$ Since $\alpha : G\rightarrow \bar{G}$ is a homomorphism, there exists a not necessarily unique homomorphism $\beta : F\rightarrow \bar{F}$ which is induced by $\alpha$ and ${\alpha}{\pi}={\bar{\pi}}{\beta}$.
\\
$(ii)$ The map $\beta$ induces a homomorphism ${\alpha}^* : \frac{R \cap V(F)}{\langle T(F)\cap R \rangle } \rightarrow \frac{\bar{R}\cap V(\bar{F})}{\langle T(\bar{F})\cap \bar{R} \rangle }$ \\
given by $x\langle T(F)\cap R \rangle \mapsto \beta(x) \langle T(\bar{F})\cap \bar{R} \rangle$, where $x\in R\cap V(F)$. We claim that ${\alpha}^*$ is well defined. Suppose that ${x_1} \langle T(F)\cap R \rangle={x_2} \langle T(F)\cap R \rangle$, where ${x_1},{x_2}\in R\cap V(F)$. We have ${x_1}{x_2}^{-1}\in \langle T(F)\cap R \rangle$ and so ${x_1}={x_2}t$, for some $t\in \langle T(F)\cap R \rangle$.
\begin{align*}
&{{\alpha}^*}({x_1}\langle T(F)\cap R \rangle )={\beta}(x_1)\langle T(\bar{F})\cap \bar{R} \rangle ={\beta}({x_2}t) \langle T(\bar{F})\cap \bar{R} \rangle \\
&={\beta}(x_2){\beta}(t) \langle T(\bar{F})\cap \bar{R} \rangle.
\end{align*}

It is enough to show that for $t\in \langle T(F)\cap R \rangle$, ${\beta}(t)\in \langle T(\bar{F})\cap \bar{R} \rangle$.\\
Since $t\in \langle R\cap T(F) \rangle$, there exist ${x_i},{y_i}\in F$ such that $t={\prod}_{i=1}^{n}[x_i,y_i]^{{\alpha}_i}$ and $[x_i,y_i]\in R\cap T(F)$. So ${\beta}(t)=\prod_{i=1}^{n}[{\beta}(x_i),{\beta}(y_i)]^{{\alpha}_i}$. Also, we have
$$\bar{\pi}({\beta}([x_i,y_i]))={\alpha}{\pi}([x_i,y_i])={\alpha}(1)=1 \ \ \  , \ \ \ {\beta}([x_i,y_i])=[{\beta}(x_i),{\beta}(y_i)]\in T(\bar{F}).$$
Therefore ${\beta}(t)\in \langle T(\bar{F})\cap \bar{R} \rangle$ and ${\alpha}^*$ is well defined.\\
Let ${\beta}'$ be another homomorphism induced by $\alpha$ such that ${\bar{\pi}}{\beta}={\alpha}{\pi}={\bar{\pi}}{\beta}'$. Then for any $f\in F$, $\beta(f)\equiv {\beta}'(f) \pmod{\bar{R}}$, and so if $v\in V$, then $v(\beta(f_1), \ldots ,\beta(f_s))\equiv v({\beta}'(f_1), \ldots ,{\beta}'(f_s)) \pmod{\langle T(\bar{F})\cap \bar{R} \rangle }$, for all $f_1, \ldots ,f_s \in F$. 
\\
Let $x\in R\cap V(F)$, then 
\begin{align*}
&x = v_1(f_1, f_2, \ldots, f_n)^{\alpha_1} \cdot v_2(f_1, f_2, \ldots, f_n)^{\alpha_2} \cdot \ldots \cdot v_n(f_1, f_2, \ldots, f_n)^{\alpha_n}\\
&=\prod_{i=1}^{n} {v_i}(f_1, \ldots ,f_{n})^{{\alpha}_i}
\end{align*}
such that  \( n \) represents the number of elements \( f_1, f_2, \ldots, f_n \) involved in the evaluation of each verbal element \( v_i \) and the exponents \( \alpha_i \) are integers. Thus we have 
\begin{align*}
&{\alpha}^*(x \langle T(F)\cap R \rangle )={\beta}(x) \langle T(\bar{F})\cap \bar{R} \rangle \\
&={\beta}(\prod_{i=1}^{n} {v_i}(f_1, \ldots ,f_{n})^{{\alpha}_i}) \langle T(\bar{F})\cap \bar{R} \rangle=\prod_{i=1}^{n} {v_i}({\beta}(f_1), \ldots ,{\beta}(f_{n}))^{{\alpha}_i} \langle T(\bar{F})\cap \bar{R} \rangle \\
&=\prod_{i=1}^{n} {v_i}({\beta}'(f_1), \ldots ,{\beta}'(f_{n}))^{{\alpha}_i} \langle T(\bar{F})\cap \bar{R} \rangle ={\beta}'(\prod_{i=1}^{n} {v_i}(f_1, \ldots ,f_{n})^{{\alpha}_i}) \langle T(\bar{F})\cap \bar{R} \rangle \\
&={\beta}'(x) \langle T(\bar{F})\cap \bar{R} \rangle ={({\alpha}')}^*(x \langle T(\bar{F})\cap \bar{R} \rangle ),
\end{align*}
where, $v_i \in V$ and ${\alpha}_i\in \Bbb{Z}$. Hence ${\alpha}^*=({\alpha}')^*$.
\\
$(iii)$ By using $(i)$ and $(ii)$, there exist homomorphisms $\bar{\beta}: \bar{F}\to \hat{F}$ and ${\gamma}^* : \mathcal{V}\tilde{\mathcal{B}}_0(\bar{G}) \rightarrow \mathcal{V}\tilde{\mathcal{B}}_0(\hat{G})$ given by $x\langle T(\bar{F})\cap \bar{R} \rangle \mapsto \bar{\beta}(x) \langle T(\hat{F})\cap \hat{R} \rangle$, where $x\in \bar{R}\cap V(\bar{F})$. By using ${\gamma}{\alpha}: G \to \hat{G}$, $(i)$ and $(ii)$, there exist two homomorphisms $\hat{\beta}=\bar{\beta}{\beta}: F\to \hat{F}$ and $({\gamma}{\alpha})^*:\mathcal{V}\tilde{\mathcal{B}}_0({G}) \rightarrow \mathcal{V}\tilde{\mathcal{B}}_0(\hat{G})$ with $x \langle T(F)\cap R \rangle \mapsto \hat{\beta}(x) \langle T(\hat{F})\cap \hat{R} \rangle $. Thus for all $x\in R\cap V(F)$, we have
\begin{align*}
&({\gamma}{\alpha})^*(x\langle T(F)\cap R\rangle)=\hat{\beta}(x)\langle T(\hat{F})\cap \hat{R}\rangle \\
&={\bar{\beta}}({\beta}(x))\langle T(\hat{F})\cap \hat{R}\rangle ={\gamma}^*(\beta(x)\langle T(\bar{F})\cap \bar{R}\rangle)={\gamma}^*({\alpha}^*(x\langle T(F) \cap R\rangle))\\
&={\gamma}^*{\alpha}^*(x\langle T(F)\cap R\rangle).
\end{align*}
 Therefore $({\gamma}{\alpha})^*={\gamma}^*{\alpha}^*$.
\\
$(iv)$ Let $G\cong {F}/{R}$ and $G\cong {\bar{F}}/{\bar{R}}$ be two presentations for $G$. Then we have
\[
\xymatrix{
1\longrightarrow R \ar[r] \ar[d] & {F } \ar[r]^{\pi} \ar[d]_{\beta} & {G}  \ar[d]^{1_G} \longrightarrow 1 \\
1\longrightarrow \bar{R} \ar[r] \ar[d] & \bar{F } \ar[r]^{\bar{\pi}} \ar[d]_{\delta} & {G}  \ar[d]^{1_G} \longrightarrow 1\\
1\longrightarrow {R} \ar[r] & {F}\ar[r]^{{\pi}}& {G} \longrightarrow 1.
&
}
\]
By using $(i)$, there exist homomorphisms $\beta$ and $\delta$, such that $\pi={\bar{\pi}}{\beta}$ and $\bar{\pi}={\pi}{\delta}$. Now using $(ii)$, there exist homomorphisms ${\alpha}^* : \frac{R\cap V(F)}{\langle T(F)\cap R \rangle }\rightarrow \frac{\bar{R}\cap V(\bar{F})}{\langle T(\bar{F})\cap \bar{R} \rangle }$ and ${\gamma}^* : \frac{\bar{R}\cap V(\bar{F})}{\langle T(\bar{F})\cap \bar{R}\rangle }\rightarrow \frac{R\cap V(F)}{\langle T(F)\cap R \rangle }$ which are independent of $\beta$ and $\delta$, respectively. Since the above diagram is commutative and ${\delta}{\beta}: F\rightarrow F$, $1_G: G\rightarrow G$ are two homomorphisms, by using $(ii)$, there exists the homomorphism $({\gamma}{\alpha})^*: \frac{R\cap V(F)}{\langle T(F)\cap R \rangle }\rightarrow \frac{R\cap V(F)}{\langle T(F)\cap R \rangle }$ given by $f \langle T(F)\cap R \rangle \mapsto {\delta}{\beta}(f) \langle T(F)\cap R \rangle $, which is independent of ${\delta}{\beta}$. Also, we can take ${\delta}{\beta}=1_F: F\rightarrow F$. So $({\gamma}{\alpha})^*=1_{\mathcal{V}\tilde{B_0}(G)}$. By using $(iii)$, $1=({\gamma}{\alpha})^*={\gamma}^*{\alpha}^*$. Similarly, we can see that ${{\alpha}^*}{{\gamma}^*}=1$. Therefore ${\alpha}^*$ and ${\gamma}^*$ are isomorphisms.
\end{proof}

As a special case, if $\mathcal{A}$ is the variety of abelian groups, then $A(F)=F'$ and $T(F)=\{[x,y] \ | \ x,y\in F\}$. So, in the finite case, the $\tilde{\mathcal{B}}_0$-invariant of $G$ is exactly the Bogomolov multiplier
$$\tilde{\mathcal{B}}_0(G)\cong \frac{R\cap F'}{\langle T(F)\cap R \rangle }.$$
If $\mathcal{N}_{c}$ is the variety of nilpotent groups of class at most $c\geq 1$, then $N_c(F)={\gamma}_{c+1}(F)$ and $T(F)=\{[x_1, \ldots ,x_{c+1}] \ | \ x_i\in F, 1\leq i\leq c+1\}$. In this case, this multiplier is called the $c$-nilpotent $\tilde{\mathcal{B}}_0$-multiplier of $G$, and defined as follows
$$\mathcal{N}_{c}\tilde{\mathcal{B}}_0(G)={\tilde{\mathcal{B}}_0}^{(c)}(G)=\frac{R\cap {\gamma}_{c+1}(F)}{ \langle T(F)\cap R \rangle}.$$
We need some elementary results about the $\tilde{\mathcal{B}}_0$-invariant. For the rest of this section, $\mathcal{V}$ is an outer commutator variety and $T(F)=\{v(f_1, \ldots ,f_s) \ | \ v\in V, f_i\in F, 1\leq i\leq s\}$.
\begin{proposition}\label{p3.3}
Let $G$ be a group and $N$ be a normal subgroup of $G$. Then the following sequence is exact
$$\mathcal{V}\tilde{\mathcal{B}}_0(G)\longrightarrow \mathcal{V}{\tilde{\mathcal{B}}_0}({G}/{N})\longrightarrow \frac{N}{\langle T(G)\cap N \rangle }\longrightarrow \frac{G}{V(G)}\longrightarrow \frac{G/N}{V(G/N)}\longrightarrow 1.$$
\end{proposition}
\begin{proof}
Let $1 \xrightarrow{} R \xrightarrow{}  F \xrightarrow{\pi} G \xrightarrow{} 1$ be a free presentation of $G$ and let \\ $K=\ker(F\to G/N)$. We have $N\cong K/R$. The inclusion maps $R\cap {V(F)}\xrightarrow{f} K\cap {V(F)},\ \ $ $K\cap {V(F)}\xrightarrow{g} K,\ \ $ $K\xrightarrow{h} F$ and $F\xrightarrow{l} F$ induce the sequence of homomorphism 

\[
\xymatrix{
\frac{R\cap {V(F)}}{\langle T(F)\cap R \rangle} \xrightarrow{f^*} \frac{K\cap {V(F)}}{\langle T(F)\cap K \rangle} \xrightarrow{g^*} \frac{K}{\langle T(F)\cap K \rangle R} \xrightarrow{h^*} \frac{F}{R{V(F)}} \xrightarrow{l^*}{\frac{F}{K{V(F)}}} \rightarrow 1
}.
\]

We can see that
$$\displaystyle{\frac{K}{\langle T(F)\cap K \rangle R}\cong \frac{N}{\langle T(G)\cap N \rangle }} \ \ , \ \  \displaystyle{\frac{F}{R{V(F)}}\cong \frac{G}{V(G)}} \ \ , \ \ \displaystyle{\frac{F}{K{V(F)}}\cong \frac{{G}/{N}}{{V(G/N)}}}.$$
Now we have 
$$\displaystyle{\mathcal{V}{\tilde{\mathcal{B}}_0}(G)\cong \frac{R\cap{V(F)}}{\langle T(F)\cap R \rangle}}\ \ \ , \ \ \ \displaystyle{\mathcal{V}{\tilde{\mathcal{B}}_0}({G}/{N})\cong\frac{K\cap{V(F)}}{\langle T(F)\cap K \rangle}}.$$
Moreover,
$$\displaystyle{\text{im }{f^*}=\ker{g^*}=} \displaystyle{\frac{R\cap{V(F)}}{\langle T(F)\cap K \rangle }}\ \ \ \ , \ \ \ \ \displaystyle{\text{im }{g^*}=\ker{h^*}=\frac{K\cap{V(F)}}{\langle T(F)\cap K \rangle R}}$$
$$\displaystyle{\text{im }{h^*}=\ker{l^*}=\frac{KV(F)}{R{V(F)}}}$$
 and $l^*$ is an epimorphism. Therefore the above sequence is exact.
\end{proof}
\begin{proposition}\label{p3.4}
Let $G$ be a group with a free presentation $G\cong {F}/{R}$, and $N$  be a normal subgroup of $G$ such that $K=\ker(F\to {G}/{N})$. Then the sequence
$$1\to \dfrac{R\ \cap \langle T(F)\cap K \rangle}{\langle T(F)\cap R \rangle} \to \mathcal{V}{\tilde{\mathcal{B}}_0}(G) \to \mathcal{V}{\tilde{\mathcal{B}}_0}({G}/{N})\to \dfrac{N\cap V(G)}{\langle T(G)\cap N \rangle} \to 1.$$
 is exact.
\begin{proof}
Suppose $1 \xrightarrow{} R \xrightarrow{}  F \xrightarrow{\pi} G \xrightarrow{} 1$ is a free presentation of $G$ and let \\ $K=\ker(F\to G/N)$. We have $N\cong K/R$. The inclusion maps\\ $R\ \cap \langle T(F)\ \cap \ K \rangle \xrightarrow{f} R\ \cap {V(F)},\ $ $R\ \cap {V(F)}\xrightarrow{g} K\cap {V(F)}$ and the map \\ $K\cap {V(F)}\xrightarrow{h} (K\cap {V(F)})R$ induce the sequence of homomorphisms 
$$
\displaystyle{1\to \frac{R\  \cap \langle T(F)\cap K \rangle}{\langle T(F)\cap R\rangle }\xrightarrow{f^*}\frac{R\cap {V(F)}}{\langle T(F)\cap R \rangle} \xrightarrow{g^*} \frac{K\cap {V(F)}}{\langle T(F)\cap K \rangle} \xrightarrow{h^*}} \displaystyle{\frac{(K\cap {V(F)})R}{\langle T(F)\cap K \rangle R} \to 1}.$$

It is straightforward to verify that 
$$\displaystyle{\langle T(G)\cap N \rangle =\frac{\langle T(F)\cap K \rangle R}{R}} \ \ \ , \ \ \ \displaystyle{N\cap {V(G)}=\frac{K}{R}\cap \frac{{V(F)}R}{R}=\frac{(K\cap {V(F)})R}{R}}.$$
Therefore we have
$$\displaystyle{\frac{N\cap {V(G)}}{\langle T(G) \cap N \rangle }=\frac{({(K\cap V(F))R})/{R}}{({\langle T(F)\cap K \rangle R})/{R}}\cong \frac{(K\cap {V(F)})R}{\langle T(F) \cap K \rangle R}}.$$
Finally we have\\
$$\displaystyle{\mathcal{V}{\tilde{\mathcal{B}}_0}(G)\cong \frac{R\cap{V(F)}}{\langle T(F)\cap R \rangle }} \ \ \ , \ \ \ \displaystyle{\mathcal{V}{\tilde{\mathcal{B}}_0}({G}/{N})\cong\frac{K\cap{V(F)}}{\langle T(F)\cap K \rangle }}$$
 and 
$$\displaystyle{\text{im }{f^*}=\ker{g^*}=\frac{R\ \cap \langle T(F)\cap K \rangle }{\langle T(F)\cap R \rangle}}\ \ \ , \ \ \ \displaystyle{\text{im }{g^*}=\ker{h^*}=\frac{R\cap{V(F)}}{\langle T(F)\cap K \rangle }}.$$
Moreover $h^*$ is an epimorphism. Therefore the above sequence is exact.
\end{proof}
\end{proposition}
\begin{lemma}\label{l3.5}
Let $G$ be a group with a normal subgroup $N$. Then there exists a group $H$ with a normal subgroup $B$ such that 
\renewcommand {\labelenumi}{(\roman{enumi})} 
\begin{enumerate}
\item{$V(G)\cap N\cong {H}/{B}$,}
\item{$B\cong \mathcal{V}\tilde{\mathcal{B}}_0(G)$,}
\item{$\mathcal{V}\tilde{\mathcal{B}}_0({G}/{N})$ is an epimorphic image of $H$.}
\end{enumerate}
\end{lemma}
\begin{proof}
Let $G\cong F/R$ be a free presentation of $G$ and suppose $N\cong K/R$. So $G/N \cong F/K$. Let 
$$H=\dfrac{V(F)\cap K}{\langle T(F)\cap R \rangle }\ \ \ \ , \ \ \ \ B=\dfrac{V(F)\cap R}{\langle T(F)\cap R \rangle }.$$
We have
\begin{align*}
V(G)\cap N &\cong \dfrac{V(F)R}{R}\cap \dfrac{K}{R}=\dfrac{V(F)R\cap K}{R}=\dfrac{(V(F)\cap K)R}{R}\\
&\cong \dfrac{V(F)\cap K}{V(F)\cap R}\cong \dfrac{(V(F)\cap K)/\langle T(F)\cap R \rangle }{(V(F)\cap R)/\langle T(F)\cap R \rangle}\\
&=\dfrac{H}{B}.
\end{align*}
Also, 
$$\mathcal{V}{\tilde{\mathcal{B}}_0}({G}/{N})\cong \dfrac{V(F)\cap K}{\langle T(F)\cap K \rangle }\cong \dfrac{(V(F)\cap K)/\langle T(F)\cap R \rangle }{\langle T(F) \cap K \rangle / \langle T(F)\cap R \rangle }.$$
Hence, $(i)$, $(ii)$ and $(iii)$ are obtained.
\end{proof}

Now we obtain an explicit formula for the $\tilde{\mathcal{B}}_0$-invariant of a direct product of two groups. 
\\
It is easy to see that if $G_1$ and $G_2$ are groups with free presentations ${F_1}/{R_1}$ and ${F_2}/{R_2}$ and $F = F_1\ast F_2$ is also a free product of $F_1$ and $F_2$, then $1\to R \to F \to G_1 \times G_2 \to 1$ is a free presentation for $G_1 \times G_2$, where $R={R_1} {R_2} [F_2,F_1]$.
 
\begin{remark}\label{r777}
By using Definition \ref{d3.1}, we have
$$\mathcal{V}{\tilde{\mathcal{B}}_0}(G_1\times G_2) = \dfrac{({R_1}{R_2}[F_2,F_1])\cap V(F)}{\langle T(F)\cap ({R_1}{R_2}[F_2,F_1]) \rangle },$$
where $V$ is the set of o.c words of the variety $\mathcal{V}$. The epimorphism $F\to F_1\times F_2$ maps $\langle T(F)\cap R \rangle$ onto ${\langle T(F_1)\cap R_1 \rangle} \oplus {\langle T(F_2)\cap R_2 \rangle}$, So there exists the following epimorphism
$$\alpha : \dfrac{R\cap V(F)}{\langle T(F)\cap R \rangle }\longrightarrow  \dfrac{R_1\cap V(F_1)}{\langle T(F_1)\cap R_1 \rangle} \oplus \dfrac{R_2\cap V({F_2})}{\langle T(F_2)\cap R_2 \rangle}$$
$$x\langle T(F)\cap R \rangle \longmapsto ({x_1} \langle T(F_1)\cap R_1\rangle \ ,\ {x_2} \langle T(F_2)\cap R_2 \rangle ),$$
where $x={x_1} {x_2} {x_3}$, $x_1\in R_1\cap V({F_1})$, $x_2\in R_2\cap V({F_2})$ and $x_3\in [F_2,F_1]$. 
\end{remark}
Before presenting Proposition 3.7, it is insightful to draw parallels with Moghaddam's seminal work on the Baer-invariant of a direct product \cite{mn}. Our proof and result share striking similarities with Lemma 2.2 therein, shedding light on the analogies between our findings and established theory.
Let $K=\ker{\alpha}$, we have the following proposition.
\begin{proposition}\label{p4.8}
Let $\mathcal{V}$ be a variety of groups defined by the set of o.c words $V$ and let $G_1$ and $G_2$ be groups. Then
$$\mathcal{V}{\tilde{\mathcal{B}}_0}(G_1\times G_2) \cong \mathcal{V}{\tilde{\mathcal{B}}_0}(G_1) \oplus \mathcal{V}{\tilde{\mathcal{B}}_0}(G_2) \oplus K.$$

\begin{proof}
Define a map
$$\beta : \dfrac{R_1\cap V({F_1})}{\langle T(F_1)\cap R_1\rangle} \oplus \dfrac{R_2\cap V({F_2})}{\langle T(F_2)\cap R_2 \rangle} \longrightarrow   \dfrac{R\cap V(F)}{\langle T(F)\cap R \rangle}$$
given by
$$ ({x_1} \langle T(F_1)\cap R_1 \rangle \ ,\ {x_2} \langle T(F_2)\cap R_2 \rangle) \longmapsto x \langle T(F)\cap R \rangle $$
is a well-defined homomorphism. It is easy to check that $\beta$ is a right inverse to $\alpha$. Thus we have a split extension and the result holds.

\end{proof}
\end{proposition}

\section{Verbal preserving extension}

In this section, we introduce a broad class of commutativity-preserving group extensions.

As $\tilde{\mathcal{B}}_0(G)$ plays a role akin to describing commutativity-preserving extensions (CP extensions), $\mathcal{V}\tilde{\mathcal{B}}_0(G)$ also serves a similar function. It is termed verbal preserving extensions (VP extensions) of groups.

\begin{definition}\label{d4.1}
Let $Q$ be a group and $N$ be a $Q$-module. The extension $1 \xrightarrow{} N \xrightarrow{\chi}  G \xrightarrow{\pi} Q \xrightarrow{} 1$ of $N$ by $Q$ is a VP extension if trivial words in elements of $Q$ have trivial lifts in $G$.
\end{definition}

\begin{definition}(\cite{14'}, Definition 3.18)\label{d4.2}
The extensions $1 \xrightarrow{} N \xrightarrow{\mu_1}  G_1 \xrightarrow{\pi_1} H \xrightarrow{} 1$ and $1 \xrightarrow{} N \xrightarrow{\mu_2}  G_2 \xrightarrow{{\pi}_2} H \xrightarrow{} 1$ are equivalent if there exists a group isomorphism $T:G_1\rightarrow G_2$ such that the following diagram is commutative

\[
\xymatrix{
1\longrightarrow N \ar[r]^{\mu_1} {\ar[d]_{1_N}} & {G_1} \ar[r]^{\pi_1} \ar[d]_{T} & {Q}  \ar[d]^{1_Q} \longrightarrow 1 \\
1\longrightarrow {N} \ar[r]^{{\mu}_2}  & {G_2} \ar[r]^{{\pi}_2}  & {Q}   \longrightarrow 1.
&
}
\]

\end{definition}

\begin{lemma}\label{l4.3}
The class of VP extensions is closed under equivalence of extensions.
\end{lemma}

\begin{proof}
Let
\[
\xymatrix{
1\longrightarrow N \ar[r]^{\mu_1} \ar[d]_{1_N} & {G_1} \ar[r]^{\pi_1} \ar[d]_{T} & {Q}  \ar[d]^{1_Q} \longrightarrow 1 \\
1\longrightarrow N \ar[r]^{{\mu}_2}  & {G_2} \ar[r]^{{\pi}_2}  & {Q}  \longrightarrow 1&
}
\]
be equivalent extensions. Suppose that $G_1$ is a VP extension of $N$ by $Q$. Choose $q_1, \ldots ,q_s\in Q$ and $v\in V$, where $v(q_1, \ldots ,q_s)=1$. Then there exist $g_1, \ldots ,g_s\in G_1$ such that $v(g_1, \ldots ,g_s)=1$ and ${{\pi}_1}(g_i)=q_i$, where $i=1, \ldots ,s$. Take ${g'_i}=T(g_i)$. Then $v(g'_1, \ldots ,g'_s)=1$ and ${{\pi}_2}(g'_i)=q_i$. Thus $G_2$ is a VP extension of $N$ by $Q$.
\end{proof}

Now, we introduce a special version of a VP extension with a marginal kernel.

\begin{definition}\label{d4.4}
The VP extension $1 \xrightarrow{} N \xrightarrow{\chi}  G \xrightarrow{\pi} Q \xrightarrow{} 1$ of $N$ by $Q$ is marginal if $\chi(N)\subseteq V^*(G)$.
\end{definition}

\begin{proposition}\label{p4.5}
Let $1 \xrightarrow{} N \xrightarrow{\chi}  G \xrightarrow{\pi} Q \xrightarrow{} 1$ be a marginal extension. This sequence is a VP extension if and only if $\chi(N)\cap T(G)=1$, where $T(G)=\{v(g_1, \ldots ,g_s) \ | \ v\in V , g_i\in G , i=1, \ldots ,s\}$.
\end{proposition}

\begin{proof}
Let $\chi(N)\cap T(G)=1$. Choose $q_1, \ldots ,q_s\in Q$ and $v\in V$ with $v(q_1, \ldots ,q_s)=1$. For some $g_i\in G$, we have $q_i=\pi(g_i)$. Then $\pi(v(g_1, \ldots ,g_s))=1$, hence $v(g_1, \ldots ,g_s)\in \chi(N)\cap T(G)$. Thus $v(g_1, \ldots ,g_s)=1$. Conversely, suppose that the above sequence is a VP marginal extension. Choose $v(g_1, \ldots ,g_s)\in \chi(N)\cap T(G)$. We know $\chi(N)=\ker {\pi}$, so $\pi(v(g_1, \ldots , g_s))=1$ and $v(\pi(g_1), \ldots ,\pi(g_s))=1$. By assumption, there exist $g'_1, \ldots ,g'_s\in G$ such that $\pi(g_i)=\pi(g'_i)$ and $v(g'_1, \ldots ,g'_s)=1$. Since $\pi(g_i)=\pi(g'_i)$, ${g_i}{g'_i}^{-1}\in \ker {\pi}=\chi(N)$. So $g_i=g'_i{a_i}$, where $a_i\in \chi(N)$. $1 \xrightarrow{} N \xrightarrow{\chi}  G \xrightarrow{\pi} Q \xrightarrow{} 1$ be a marginal extension. Thus we have $v(g_1, \ldots ,g_s)=v(g'_1a_1, \ldots ,g'_sa_s)=v(g'_1, \ldots ,g'_s)=1$, as required.
\end{proof}

\begin{definition}\label{d4.6}
The normal subgroup $N$ of a group $G$ is a VP subgroup of $G$ if the extension $1 \rightarrow{} N \rightarrow{}  G \rightarrow{} G/N \xrightarrow{} 1$ is a VP extension.
\end{definition}

Note that, when $N$ is a marginal subgroup of $G$, Proposition \ref{p4.5} shows that $N$ is a VP subgroup if and only if $N\cap T(G)=1$

.

Let $1 \rightarrow{} R \rightarrow{} F \rightarrow{} G \rightarrow{} 1$ be a free presentation of $G$. Then the extension 
\[
1 \rightarrow{}\frac{R \cap V(F)}{\langle T(F)\cap R\rangle}\rightarrow{}\frac{R}{\langle T(F)\cap R\rangle}\rightarrow{}\frac{R}{R \cap V(F)}\rightarrow{} 1
\]
is a VP extension. So, the normal subgroup $\frac{R \cap V(F)}{\langle T(F)\cap R\rangle}$ of $\frac{R}{\langle T(F)\cap R\rangle}$ is a VP subgroup. Also for all $r_i \in R \ (i=1,\ldots ,s)$, $a\in R\cap V(F)$ and for all $v\in V$, we have
\[
v(r_1,\ldots , r_i{a},\ldots ,r_s)v(r_1,\ldots ,r_i,\ldots ,r_s)^{-1} \in \langle T(F)\cap R\rangle.
\]
Thus, $\frac{R \cap V(F)}{\langle T(F)\cap R\rangle}$ is a marginal subgroup of $\frac{R}{\langle T(F)\cap R\rangle}$.

\begin{lemma}\label{l4.7}
Let $N$ be a marginal VP subgroup of $G$. Then the sequence
\[
1 \xrightarrow{} \mathcal{V}\tilde{\mathcal{B}}_0(G) \xrightarrow{}  \mathcal{V}\tilde{\mathcal{B}}_0(G/N) \xrightarrow{} N\cap V(G) \xrightarrow{} 1
\]
is exact.
\end{lemma}

\begin{proof}
Let $G\cong F/R$ be a free presentation of $G$ and suppose $K=\ker (F\rightarrow G/N)$. We have $N\cong K/R$, so $G/N\cong F/K$. Since $N$ is a marginal VP subgroup of $G$, then $\langle T(F)\cap K \rangle \leq R$. Thus, $\langle T(F)\cap K \rangle= \langle T(F)\cap R \rangle $. Now the result is obtained by using Lemma \ref{l3.5}.
\end{proof}

\section*{Declarations}

\begin{itemize} 
\item Funding\\
The authors did not receive support from any organization for the submitted work.
\item Conflict of interest\\
The authors have no competing interests to declare that are relevant to the content of this article.
\item Ethics approval \\
Not applicable.
\item Consent to participate\\
Informed consent was obtained from all individual participants included in the study.
\item Consent for publication\\
Not applicable.
\item Availability of data and materials\\
Not applicable.
\item Code availability \\
Not applicable.
\item Authors' contributions\\
All authors contributed almost equally.
\end{itemize}

\end{document}